\documentclass[12pt]{amsart}
\usepackage{amscd, amsmath, amsfonts, amsthm , amssymb}

\newtheorem{thm}{Theorem}
\newtheorem{prop}{Proposition}
\newtheorem{lem}{Lemma}
\newtheorem{cor}{Corollary}

\theoremstyle{remark}
\newtheorem{rem}{Remark}

\newtheorem{ex}{Example}

\newcommand{\Q}{\mathbb{ Q}}
\newcommand{\N}{\mathbb{ N}}
\newcommand{\R}{\mathbb{ R}}

\newcommand{\rk}{\operatorname{rk}}

\newcommand{\Z}{\mathbb{ Z}}

\DeclareMathOperator{\LL}{\mathfrak{L}}

\DeclareMathOperator{\im}{Im}
\DeclareMathOperator{\Ker}{Ker}

\title[Geometric formality of rationally elliptic manifolds]{Geometric formality of rationally elliptic manifolds in small dimensions}
\author{Svjetlana Terzi\'c}
\address{Faculty of Science, University of Montenegro, D\v zord\v za Va\v singtona bb, 81000 Podgorica,
Montenegro}
\address{Montenegrin Academy of Sciences and Arts, Rista Stijovi\'ca 5, 81000 Podgorica, Montenegro}
\email{sterzic@ac.me}
\begin{document}

\maketitle

\begin{abstract}
We classify  simply connected rationally elliptic manifolds of dimension five and those of dimension six with  small Betti numbers from the point of view of their rational cohomology structure. We also prove that a geometrically formal rationally elliptic six dimensional manifold, whose second Betti number is two, is rational cohomology $S^2\times {\mathbb C}P^2$. An infinite family of six-dimensional simply connected biquotients whose second Betti number is three, different from Totaro's biquotients, is  considered and it is proved that none of  biquotient from this family is geometrically formal.  
\end{abstract}
\vskip10mm

MSC:\; 53C25, 53C30

\section{Introduction}
Geometric formality of a compact smooth manifold $M$ is the notion introduced in~\cite{kot} and it is concerned with an existence of a Riemannian metric on $M$ such that the related harmonic forms form an algebra. Such a metric is called formal. Geometrically formal manifolds are formal in the sense of rational homotopy theory. Moreover, the original proof of the rational formality of symmetric spaces, which are one of the first non-trivial examples of such spaces, implicitly used the fact that the invariant metric on symmetric spaces is formal. But, it turns out that the notion of geometric formality is  much more restrictive. It is proved in~\cite{kot} that in dimension less or equal four, a geometrically formal manifold must have real cohomology of a symmetric space. Besides that, in~\cite{KT1} are provided a lot examples of homogeneous spaces which are rationally formal and which from cohomological reasons are not geometrically formal. 

Therefore, the investigation of geometric formality  in general, as well as for  some family of spaces or particular  examples remains unsolved, interesting and important problem for many applications in pure mathematics and mathematical physics. In the recent time as well in the focus of mathematical interest is the connection between  geometric formality
property and different differential geometrical properties of manifolds such as sectional and scalar curvature~\cite{K1},~\cite{B}.
 
As noted, this problem should be studied in the category of rationally formal spaces. Besides that, the Hodge theory gives that for geometrically formal spaces the algebra structure of the harmonic forms is the same as its real cohomology algebra structure. It suggests that one can hardly expect the positive answer to the question of geometric formality for the manifolds with many cohomology relations. In this paper, from these reasons,  we study rationally elliptic manifolds in the next two unsolved dimensions five and six.  For these manifolds  it is  known to comprise simply connected homogeneous spaces and biquotients in corresponding dimensions.

\section{General background}

\subsection{ Rational homotopy theory}

We refer to~\cite{FHT}  for a  comprehensive general reference for rational homotopy theory.

Let $(\mathcal{A},d_{\mathcal{A}})$ be a connected $(H^{0}(\mathcal{A}, d_{\mathcal{A}}) = k)$
and simply connected ($H^{1}(\mathcal{A}, d_{\mathcal{A}}) =0$) commutative $\N$-graded differential algebra over a field $k$ of characteristic zero. Let us consider the free $\N$-graded commutative differential algebra $(\wedge V, d)$ for a $\N$-graded vector space $V$ over $k$. It is said that
$(\wedge V, d)$ is a minimal model for $(\mathcal{A},d_{\mathcal{A}})$
if $d(V)\subset \wedge ^{\geq 2}V$ and there exists a morphism
$$
f : (\wedge V, d)\to (\mathcal{A}, d_{\mathcal{A}}) \ ,
$$
which induces an isomorphism in cohomology.

Let $X$ be a simply connected topological space of finite type.  The minimal model $\mu (X)$ for $X$ is defined to be to be the minimal model for the
algebra $\mathcal{A} _{PL}(X)$ of piece-wise linear forms on  $X$. One says that two simply connected manifolds have the same rational homotopy type if and only if there is a third space to which
they both map by maps inducing isomorphism in rational cohomology.   Then the following facts are well known. The minimal model $\mu (X)$ of a  simply connected topological space $X$ of finite type is unique up to isomorphism (which is well defined up to homotopy), it classifies the rational homotopy type of $X$ and,
furthermore, it contains complete information on the ranks of the homotopy groups of $X$. More precisely,
\begin{equation}\label{ranks}
\rk \pi _{r}(X) = \dim (\mu (X)/\mu ^{+}(X)\cdot \mu ^{+}(X))_{r}, \;
r\geq 2 \ ,
\end{equation}
where by $\mu ^{+}(X)$ we denote the elements in $\mu (X)$ of
positive degree and $\cdot$ is the usual product in $\mu (X)$.
One of the equivalent definition of formality is that $X$ is formal in the sense of Sullivan if its minimal model
coincides with the minimal model of its cohomology algebra $(H^{*}(X,\Q ), d=0)$ (up to isomorphism).

The procedure for minimal model construction is given, see~\cite{FHT}, through the proof of the theorem which  states the existence (and also the uniqueness up to isomorphism) of the minimal model for any such algebra. We briefly describe this procedure here, since we are going to apply it explicitly.

\subsubsection{ Procedure for minimal model construction.}\label{PR}
In the procedure for the construction  of the  minimal model for a simply connected commutative differential
$\N$-graded algebra $(\mathcal{A}, d)$ one starts by choosing
$\mu _{2}$ and
$m_{2} : (\mu _{2}, 0)\to (\mathcal{A}, d)$
such that
$m_{2}^{(2)} : \mu _{2}\to H^{2}(\mathcal{A}, d)$
is an isomorphism. In the inductive step, supposing that
$\mu_{k}$ and $m_{k} : (\mu _{k}, d)\to (\mathcal{A}, d)$
are constructed we extend it to $\mu _{k+1}$ and
$m_{k+1} : (\mu _{k+1}, d)\to (\mathcal{A}, d)$ with
\begin{equation}\label{MM}
\mu _{k+1}=\mu _{k}\otimes \mathcal{L}(u_{i},v_{j}) \ ,
\end{equation}
where $\mathcal{L}(u_{i},v_{j})$  denotes the vector space spanned by the elements
$u_{i}$ and $v_{j}$ corresponding to $y_{i}$ and $z_{j}$ respectively.
The latter are given by
\begin{equation}\label{M1}
H^{k+1}(\mathcal{A})=\Im m_{k}^{(k+1)}\oplus \mathcal{L}(y_{i})
\end{equation}
and
\begin{equation}\label{M2}
\Ker m_{k}^{(k+2)}=\mathcal{L}(z_{j}) \ .
\end{equation}
Then we have that $m_{k}(z_{j})=dw_{j}$ for some $w_{j}\in \mathcal{A}$ and the homomorphism $m_{k+1}$ is defined by
$m_{k+1}(u_{i})=y_{i}$, $m_{k+1}(v_{j})=w_{j}$ and $du_{i}=0$, $dv_{j}=z_{j}$.

\begin{rem}
In general, for a simply connected topological space $X$ we have that  $ \mathcal{A} = \mathcal{A}_{PL}(X)$ and, obviously, by~\eqref{ranks}, we see that   $\rk \pi _{k+1}(X)$ is the number of generators in the above procedure we add to $\mu _{k}(X)$, in order to obtain $\mu _{k+1}(X)$ .
\end{rem}
\begin{rem}\label{good}
For some spaces with special  cohomology one can easily compute
their minimal models. Namely, assume that  the rational cohomology algebra for $X$ is given by

$$
H^{*}(X,\Q )\cong \Q [x_1,\ldots ,x_n]/\langle P_1,\ldots ,P_k \rangle \ ,
$$
where the polynomials $P_1,\ldots ,P_k$ are without relations
in $\Q [x_1,\ldots ,x_n]$ meaning that $\langle P_1,\ldots ,P_k\rangle$ is a Borel ideal.
Then in~[2] it is proved that such a space $X$ is formal and its
minimal model is given by
$$
\mu (X)=\Q [x_1,\ldots ,x_n]\otimes \wedge (y_1,\ldots ,y_k) \ ,
$$
$$
dx_i = 0, \;\; dy_i =P_i \ .
$$
\end{rem}

Note that for a formal $X$, the algebra  $\mu (X)\otimes _{\Q}k$ coincides with the minimal model of the cohomology algebra $(H^{*}(X,k), d=0)$ for any field $k$ of characteristic zero. The converse is also true. If there exists a field $k$ of characteristic zero for which $\mu (X)\otimes _{\Q}k$ is the minimal model for the cohomology algebra $(H^{*}(X,k), d=0)$, then $X$ is formal. By the result of~\cite{MN}, all simply connected manifolds of dimension $\leq 7$ are formal in the sense of rational homotopy theory.

\begin{rem}
Obviously,~\eqref{ranks} implies that for the purpose of calculating the ranks of the homotopy groups of $X$ we can use $\mu (X)\otimes _{\Q}\R$ as well.
In the case of formal $X$ it means that we can apply the above procedure to
$H^{*}(X,\R )$.
\end{rem}

\subsubsection{ Rationally elliptic spaces.}  Suppose that $X$ is a simply connected topological space with rational homology of finite type. It is  said that $X$ is {\it rationally elliptic} if $\dim \pi _{*}(X)\otimes \Q $ is finite. 
The ranks of the homotopy groups of a rationally elliptic space $X$ of dimension $n$ satisfy~\cite{FHT}:
\begin{equation}\label{ranks-formulas}
\sum _{k}2k\rk \pi _{2k}(X) \leq n,\;\;\; \sum _{k}(2k+1)\rk \pi _{2k+1}(X)\leq 2n-1 .
\end{equation}
\begin{ex} 
For the spaces whose cohomology is given by Remark~\ref{good} we see that~\eqref{ranks} implies that they are rationally elliptic. Among the examples of such spaces are compact homogeneous spaces of positive Euler characteristic and biquotients of compact Lie groups.
\end{ex}

\subsection{Geometric formality}

The notion of geometric formality is introduced in~\cite{kot}. A smooth manifold $M$ is said to be {\it geometrically formal} if it admits Riemannian metric for which the wedge product of any two harmonic forms is again harmonic form. Recall~\cite{chern} that a differential form $\omega \in \Omega _{DR}(M)$ is harmonic if $\Delta \omega =0$, where $\Delta$ is the Laplace-de Rham operator on $\Omega _{DR}(M)$. Using Hodge theory it is proved~\cite{chern} that any harmonic form is closed and that no harmonic form is exact. Moreover any  real cohomology class for  $M$ contains  unique, up to constant, harmonic representative. It implies that manifolds which are rational homology spheres are trivially geometrically formal, they have just one, up to constant, harmonic form. The Hodge theory also implies that a geometrically  formal manifold is formal in the sense of rational homotopy theory. The vice versa is not true, it turns out that the notion of geometric formality is much more restrictive  then the  rational formality notion. 

The first non-trivial and up to now the widest class of examples  of geometrically formal manifolds are compact symmetric spaces~\cite{DFN}. They are as well one of the first examples of rationally formal spaces and the proof of their formality is based on the fact that on symmetric spaces harmonic forms related to an invariant metric form an algebra.    

 The non-involutive symmetries of higher order do not have any more such  properties. It is proved in~\cite{KT1}  that the generalized symmetric spaces are formal, while for most of them it is proved that they are not geometrically formal. The examples of non-symmetric geometrically formal homogeneous spaces that are not cohomology sphere are provided in~\cite{KT2}.

\section{ Rationally elliptic manifolds of dimension five of six}

We consider  rationally elliptic manifolds of dimension $5$
and $6$. In dimension $5$ we prove the following.

\begin{prop}
Let $M$ be a smooth, compact, simply connected five-dimensional rationally elliptic manifold. Then $M$ is  
rational cohomology sphere $S^5$ or it is rational cohomology product of spheres $S^2\times S^3$.
\end{prop}
\begin{proof}
Since $M$ is simply connected and rationally elliptic for its Betti numbers it holds that  
$b_{4}(M)=b_{1}(M)=0$ and $2b_{2}(M) + 4\rk \pi_{4}(M)\leq 5$ and $3\rk \pi _{3}(M)+5\rk \pi _{5}(M)\leq 9$. Therefore we must have
\begin{equation}\label{restr}
b_{2}\leq 2\;\; \text{and}\;\; \rk \pi _{3}(M)\leq 3.
\end{equation}

Now using described  procedure for minimal model construction we obtain that $\mu _{2} = H^{2}(M)$ and  $m_{2}^{(3)}: H^{3}(\mu _{2},\R )\to H^{3}(M,\R)$, $m_{2}^{(4)}: H^{4}(\mu _{2},\R )\to H^{4}(M, \R)$. Since $H^{3}(\mu _{2},\R ) =H^{4}(M, \R )=0$ it follows that $\im m_{2}^{(3)} = 0$ and $\Ker m_{2}^{(4)}=\mu _{2}^{4}$. 
It implies that $\mu _{3} = \mu _{2}\otimes \LL (u_i, v_j)$, where $u_i$ correspond to the basis $y_i$ in $H^{3}(M, \R)$, while $v_j$ correspond to the basis $z_j$ in $\mu _{2}^4$. The homomorphism $m_{3}$ is an extension of $m_{2}$ defined by $m_{3}(u_i)=y_i$ and $m_{3}(v_j)=0$.   It in particular gives that  
\begin{equation}\label{ranks}
\rk \pi _{3}(M) = b_{3}(M)+\dim \mu _{2}^{4} = b_{2}(M)+\frac{b_{2}(b_{2}+1)}{2} .
\end{equation}    
Together with~\eqref{restr} it implies   $b_2(M)=0$ or $b_{2}(M)=1$.

For  $b_{2}(M)=0$ it immediately follows  that $M$ is rational cohomology sphere $S^5$.

For $b_{2}(M)=1$, using Poincar\'e duality  
we deduce that $M$ is rational cohomology product of spheres $S^2\times S^3$. 
\end{proof}
Any  six-dimensional  simply-connected rationally elliptic space  $M$ satisfies  $b_{1}(M)=b_{5}(M)=0$ and    
\[
2b_{2}(M)+ 4\rk \pi_{4}(M)+6\rk \pi _{6}(M)\leq 6, \;\;
3 \rk \pi _{3}(M)+5\rk \pi _{5}(M)\leq 11.
\]
This implies that 
\begin{equation}\label{ranks6}
b_2\leq 3,\;\;\; \rk \pi _{4}(M)\leq 1,\;\;   \rk \pi_{3}(M)\leq 3.
\end{equation} 

When the second Betti number is one or zero we prove:
\begin{prop}
Let $M$ be a smooth, compact, simply-connected and rationally elliptic six-manifold whose second Betti number is less or equal one. Then $M$ is rational cohomology sphere $S^6$, the product of spheres $S^2\times S^4$ or $S^3\times S^3$, the complex projective space ${\mathbb C}P^3$.
\end{prop}
\begin{proof}

If $b_2=0$ then $b_{4}=0$ and  Hurewitz theorem gives that $\rk \pi_{3}=b_3$. Since, by Poincar\'e duality, $b_{3}(M)$ has to be even we might have $b_3=0$ or $b_3=2$. For $b_{3}(M)=0$ all Betti numbers $b_{i}(M)=0$, $1\leq i\leq  4$ are trivial what implies that $M$ is rational cohomology sphere $S^6$. For $b_{3}=2$ the Poincar\'e duality implies that   $M$ is rational cohomology product of spheres $S^3\times S^3$. 

For  $b_{2}=1$ we claim  that  $b_3=0$. First $\mu _{3}=\LL (x)\otimes \LL (u_i,v)$, where $x$ is the generator for $H^{2}(M,\R )$, then $u_i$ correspond to the basis for $H^{3}(M)$ and $v$ corresponds to the generator $z$  for $\Ker m_{2}^{4}$. Also the differential $d$ in $\mu _{3}$ is given by $d(u_i)=0$ and $d(v) = z$. It means that $\rk \pi _{3}(M)\geq b_{3}(M)$ implying $b_{3}(M)\leq 2$.  
For $b_{3}(M)=2$, since $b_{5}(M)=0$,  we would have that $\Ker m_{3}^{(5)} = H^{5}(\mu _{3}, d) = \LL (xu_1,xu_2)$ is two-dimensional what would imply that 
$\rk \pi _{4}(M)\geq 2$. This contradicts with~\eqref{ranks6}.

Thus, let $b_2=1$ and $x\in H^{2}(M,\R)$. If  $x^{2}=0$,
we  have in $H^{*}(M)$ a generator
of degree $4$ and thus in this case $M$ is rational cohomology 
$S^2\times S^4$.  If $x^{2}\neq 0$,  then 
$M$ is rational cohomology complex projective space ${\mathbb C} P^{3}$.
\end{proof}

\section{Geometric formality and rational ellipticity in dimension six} 
We consider in this section geometrically formal six-dimensional rationally elliptic manifolds whose second Betti number is $2$ or $3$. For those whose second Betti number is $2$ we obtain rational cohomology description, while among those whose second Betti number is $3$, we study  the class of homogeneous spaces and biquotients.  

\begin{thm}\label{main}
Let $M$ be a smooth, compact, simply connected rationally elliptic six-manifold such whose second Betti number is $2$. If $M$ is geometrically formal then
$M$ is rational cohomology $S^2\times {\mathbb C}P^2$.
\end{thm}

\begin{proof}

Let  $x$ and $y$ be the generators in
$H^{*}(M)$ of degree 2. By~\eqref{ranks6} we have that  $\rk \pi_{4}(M)=0$ what implies that there is no generator of degree four in the cohomology ring $H^{*}(M,\R )$.  Since $b_{2}(M)=b_{4}(M)$ it further gives that  there is  exactly one relation in $H^{*}(M)$ of degree 4. Therefore $\mu _{3}=H^{2}(M, \R )\otimes \LL (u_i,v)$, where $u_i$ correspond to the basis $y_i$ for $H^{3}(M)$ while $v$ correspond to the non-zero element $z$ from $\Ker m_{2}^{(4)}$.  We also have that $m_{3}(u_i)=y_i$, $m_{3}(v)=0$ and the differential $d$ in $\mu _{3}$ is given by $d(u_i)=0$, $d(v)=z$.  Thus   $\im m_{3}^{(4)} = H^{4}(M, \R)$ and $\Ker m_{3}^{(5)} = \LL (xu_i, yu_i)$, where $x,y$ are the generators for $H^{2}(M, \R )$  which means that  the dimension of $\Ker m_{3}^{(5)}$ is  $2b_{3}(M)$. It  gives that $\mu _{4}=\mu _{3}\otimes \LL (w_i)$, where $w_i$ correspond to the generous for $\Ker m_{3}^{(5)}$ and consequently  $\rk \pi _{4}(M) = 2b_{3}(M)$. Therefore we obtain $b_{3}(M)=0$. The elements $x^3, x^2y, xy^2, y^3$ are in $H^{6}(M,\R)$ which is one-dimensional. Two relations among these elements come form the relation in $H^{4}(M,\R )$ and, thus, me must have exactly one new relation in $H^{6}(M,\R )$.    

The following cases are possible.

a) There exist generator of degree $2$ whose square is zero. In this case Poincar\'e duality implies that $M$ has
cohomology of $S^2\times {\mathbb C}P^2$.

b) There is no  generator of degree $2$ whose square is zero. In this case we prove that  one can always find generators
${\bar x}$ and ${\bar y}$ for  $H^{2}(M)$ such that
\begin{equation}\label{gen-rel}
{\bar x}^2 + \epsilon{\bar y}^{2}=0\;\; \text{and}\;\; {\bar y}^{3}=0,\;\; \text{for}\;\;  \epsilon=\pm 1.
\end{equation}
It would thehn imply that  ${\bar x}^2y=0$ and ${\bar x}^3\neq 0$.

In the cohomology ring $H^{*}(M,\R )$ there are exactly two independent relations, one in  degree four and the other one in degree six. The relation in degree four is of the form $ax^2+bxy+cy^2=0$, where $a^2+b^2+c^2\neq 0$. We differentiate the two cases.

1) If $a=c=0$ this relation writes as $xy=0$. It implies that $x^2y = xy^2=0$ and $x^2$, $y^2$ are going to be Poincar\'e duals to $x$ and $y$ respectively what gives $x^3,y^3\neq 0$. Thus,the relation of degree six is of the form $y^3=ax^3$, $a\neq 0$. Put  $x_1=\sqrt[3]{a}x$ and consider the new generators for $H^{*}(M,\R )$ given by ${\bar x}= x_1+y$ and ${\bar y}=x_1-y$. Then ${\bar x}^2 -{\bar y}^2=0$ and ${\bar y}^3 = x_1^3-y^3=0$, satisfying~\eqref{gen-rel}.

2) If $a^2+c^2\neq 0$ let us, without loss of generality assume that $a\neq 0$. The relation of degree four writes as 
$x^2+bxy +cy^2=0$ for some new $b$ and $c$ what gives $(x+\frac{b}{2}y)^2 +(c-\frac{b^2}{4})y^2=0$.  Note that $c-\frac{b^2}{4}\neq 0$ since otherwise we would have the generator $x_1=x+\frac{b}{2}y$ whose square is zero. Thus,   $x_1$ and $y$ satisfy in degree four the relation $x_1^2+by^{2}=0$ for some new $b\neq 0$. If we further take $y_1=\sqrt{\left|b\right|}y$ we obtain the generators $x_1$ and $y_1$ for $H^{*}(M,\R )$ related in degree four by $x_1^2\pm y_1^2=0$. Without loss of generality assume that $x_1^2+y_1^2=0$. It implies $x_1^3=-x_1y_1^2$ and $y_1^3=-x_1^2y_1$. If $x_1^3=0$ or $y_1^3=0$ we take ${\bar x}$ and ${\bar y}$ to be $x_1$ and $x_2$, and~\eqref{gen-rel} will be  satisfied. If $x_1^3,y_1^3\neq 0$ we will show that for some $a\in \R$ the cube of  element  $x+ay$ has to be zero. Namely, taking into account relations between $x$ and $y$ we obtain
\[
(x+ay)^3=(1-3a^2)x^3+(a^3-3a)y^3.
\]
Further in this case the relation in degree six writes as $y^3=\alpha x^3$ what gives
\[
(x+ay)^3 = (1-3a^2+\alpha(a^3-3a))x^3.
\]
Therefore $(x+ay)^3=0$ if and only if $1-3a^2+\alpha (a^3-3a)=0$. The later one equation, being cube equation in $a$, always has at least one solution. Take ${\bar x}=ax-y$ and ${\bar y}=x+ay$. Then
\[
{\bar x}^2+{\bar y}^2 = (a^2+1)(x^2+y^2)=0,\;\; {\bar y}^3=0.
\]

We continue the proof. If we assume $M$ to be geometrically formal,  the relations ~\eqref{gen-rel} between the cohomology classes will be satisfied for the harmonic forms $\omega _1$ and $\omega _2$ representing $x$ and $y$. It would imply that the form $\omega _2$ has non trivial kernel meaning that there locally exists non-trivial vector field $v$ such that $i_{v}\omega _2=0$.  It would also follow that $\omega _1^3$ is a volume form on $M$.
But then from~\eqref{gen-rel} we deduce that $i_{v}(\omega _1^2)=0$  as well implying that $i_{v}(\omega _1)\omega _1=0$. Therefore
\[
i_{v}(\omega _1^3)=3i_{v}(\omega _1)\omega _1^2 =0
\]
what is in the contradiction with the fact that $\omega _1$ is a volume form on $M$. Thus such $M$ can not be geometrically formal.
\end{proof}

\begin{rem}  
Note that a  manifold $M$ which satisfies conditions of Theorem~\ref{main} and which is not rational cohomology $S^2\times {\mathbb C}P^2$ can not be geometrically formal for cohomological reasons. It implies that such $M$ may not have the cohomology of a  symmetric space.
\end{rem}

\subsection{ Homogeneous spaces and biquotients.}

It is known~\cite{FHT}  that homogeneous spaces and biquotinets of a compact Lie group are rationally elliptic. Together with Theorem~\ref{main} it implies:

\begin {cor}\label{hom-biq}
Simply connected six-dimensional homogeneous space or biquotient whose  second Betti number is two and which is not cohomology $S^2\times {\mathbb C}P^2$  can not geometrically formal.
\end{cor}

Examples of spaces which satisfy the conditions in Corollary~\ref{hom-biq} are, among the others, flag manifold $SU(3)/T^2$ and  Eschenburg's biquotients~\cite{E}. Therefore these spaces are not geometrically formal. We want to note that  for the flag manifold and some of Eschenburg's biquotients  it is proved in~\cite{KT1} and~\cite{KT2}  that they are not geometrically formal. It is done treating separately  each of these examples by  studying their  cohomology  structure. By Corollary~\ref{hom-biq} we provide general proof for all of them.

\begin{lem}
Let $M$ be a simply connected rationally elliptic six-dimensional manifold . If $b_{2}(M)=3$ then $M$ has the rational homotopy groups of $S^2\times S^2\times S^2$.
\end{lem}
\begin{proof}
The assumption that $b_{2}(M)=3$ and that $M$ is rationally elliptic implies that $\rk \pi _{2}(M)=3$, and $\rk \pi _{2k}(M)=0$ for $k\geq 2$.    We further obtain that $\im m_{2}^{(3)}=0$ and since $b_{4}(M)=3$ we must have $\dim \Ker m_{2}^{(4)}\geq 3$, what implies that $\rk \pi _{3}\geq 3$. It then follows from~\eqref{ranks} that $\rk \pi _{3}(M)=3$ and $\rk \pi _{2k+1}(M)=0$ for $k\geq 2$. Thus $M$ has the rational homotopy groups of $S^2\times S^2\times S^2$. 
\end{proof}

By the recent results obtained in~\cite{T} and~\cite{DV}  it follows that  a biquotinent $M$ which has the same rational homotopy groups as $S^2\times S^2\times S^2$ is diffeomorphic to the biquotient of $SU(2)\times SU(2)\times SU(2)$ by some free linear action of $S^1\times S^1\times S^1$. On the other hand the result of~\cite{DV}  says that any free linear action of $T^3$ on $(S^{3})^3$, up to reparametrization, is given by
\[
(u,v,w)\ast ((p_1,p_2),(q_1,q_2),(r_1,r_2)) = 
\]
\[
((up_1,u^{a_1}v^{a_2}w^{a_3}p_2),(vq_1,u^{b_1}v^{b_2}w^{b_3}q_2),(wr_1,u^{c_1}v^{c_2}w^{c_3}r_2)),
\]
where $a_1,b_2,c_3=\pm 1$ and the $2\times 2$ minors around the diagonal of the following matrix 
\[
\left[\begin{array}{ccc}
a_1 & a_2 & a_3\\
b_1 & b_2 & b_3\\
c_1 & c_2 & c_3
\end{array}\right]
\]
as well as the matrix itself have determinant $\pm 1$. 
It was shown in~\cite{T}  that the conditions on the entries of this matrix  are necessary and sufficient conditions for this action to be free. In~\cite{DV} all such matrices are classified and there are obtained, up to equivalences, three infinite families of matrices  and $12$ sporadic examples. The infinite families of matrices are:
\[
\left[\begin{array}{ccc}
1 & 2 & 0\\
1 & 1 & 0\\
c_1 & c_2 & 1
\end{array}\right],\;\;
\left[\begin{array}{ccc}
1 & 2 & a_3\\
1 & 1 & b_3\\
0 & 0 & 1
\end{array}\right],\;\;
\left[\begin{array}{ccc}
1 & 0 & 0\\
b_1 & 1 & 0\\
c_1 & c_2 & 1
\end{array}\right].
\]
The action of $T^3$ on $(S^{3})^{3}$ given by the second matrix produces Totaro's biquotients studied in~\cite{T} as an example  of family of $6$-manifolds with nonnegative sectional curvature, but with infinitely many distinct classes of rational cohomology rings.  For them, it is proved in~\cite{KT2}  not to be geometrically formal for cohomological reasons.

We prove here that the  family of biquotients given by the third matrix is not geometrically formal for cohomological reasons.
\begin{thm}
The biquotients $(S^{3})^3/T^3$ given by the action
\[
(u,v,w)\ast ((p_1,p_2),(q_1,q_2),(r_1,r_2)) = 
\]
\[
((up_1,up_2),(vq_1,u^{b_1}vq_2),(wr_1,u^{c_1}v^{c_2}wr_2)),
\]
are not geometrically formal.
\end{thm}
\begin{proof}
Using the standard techniques~\cite{T} for computing cohomology rings one deduce~\cite{DV} that any  biquotient $M$ obtained by the given action has the following integral cohomology structure :
\[
H^{*}(M)=\Z [x_1,x_2,x_3]/\left\langle x_1^2=0,x_2^2=-b_1x_1x_2, x_3^2=-c_1x_1x_3-c_2x_2x_3\right\rangle.
\]
Let us assume that $M$ admits a formal metric.  Denote by $\omega _{1}, \omega _{2}$ and $\omega _{3}$  the corresponding harmonic  representatives  of the   
real cohomology generators $x_1,x_2$ and $x_3$. Then these harmonic forms will satisfy  the relations which hold in  $H^{*}(M,\R)$ between $x_1,x_1$ and $x_3$:
\begin{equation}\label{relhar}
\omega _{1}^{2}=0,\;\; \omega _{2}^2=-b_1\omega _{1}\omega _{2},\;\; \omega _{3}^2 = -c_1\omega _{1}\omega _{3} - c_2\omega_{2}\omega _{3}.
\end{equation}
It implies that $(\frac{b_1}{2}\omega _{1}+\omega _{2})^2=0$ and we further consider the harmonic form ${\tilde \omega _{2}}=\omega _{2}+\frac{b_1}{2}\omega _{1}$ which satisfies  ${\tilde \omega _{2}}^2=0$. Thus
$\omega _{3}^2=(\frac{b_1}{2}c_{2}-c_{1})\omega _{1}\omega _{3}-c_{2}{\tilde \omega _{2}}\omega _{3}$ and we consider  the form
\[
{\tilde \omega _{3}}=\omega _{3}-\frac{1}{2}(\frac{b_1}{2}c_2-c_1)\omega _{1}+\frac{1}{2}{\tilde \omega _{2}},
\]
whose square satisfies
\begin{equation}\label{three}
{\tilde \omega _{3}}^{2}=-\frac{c_2}{2}(\frac{b_1}{2}c_2-c_1)\omega _{1}{\tilde \omega _{2}}.
\end{equation}
Therefore the form $\omega _{1}{\tilde \omega _{2}}{\tilde \omega _{3}}$ is a volume form on $M$ and such is 
${\tilde \omega _{3}}^3=-\frac{c_2}{2}(\frac{b_1}{2}c_2-c_1)\omega _{1}{\tilde \omega _{2}}{\tilde \omega _{3}}$ as well.

The kernel foliations for the forms $\omega _{1}$ and ${\tilde \omega _{2}}$ are each of dimension at least four, so they have  a common vector field $v$ (more precisely at least two) meaning that $i_{v}(\omega _{1})=i_{v}({\tilde \omega _{2}})=0$. Together with~\eqref{three} it implies 
\[
i_{v}({\tilde \omega _{3}}^2)=2i_{v}({\tilde \omega _{3}}){\tilde \omega _{3}} = 0
\] 
and, thus,
\[
i_{v}({\tilde \omega _{3}}^3)=3i_{v}({\tilde \omega _{3}}){\tilde \omega _{3}}^2 = 0,
\]
what is in contradiction with the fact that ${\tilde \omega _{3}}^3$ is a volume form on $M$.
\end{proof}
\bibliographystyle{book}

\end{document}